\documentclass{article}
\usepackage{amssymb,latexsym}
\usepackage{amsmath}
\usepackage{graphics}
\usepackage{graphicx}
\newtheorem{theorem}{Theorem}

\newtheorem{corollary}[theorem]{Corollary}
\newtheorem{definition}[theorem]{Definition}
\newtheorem{example}[theorem]{Example}

\newtheorem{proposition}[theorem]{Proposition}
\newtheorem{lemma}[theorem]{Lemma}
\newtheorem{remark}[theorem]{Remark}
\newenvironment{proof}[1][Proof]{\noindent\textbf{#1.} }{\ \rule{0.5em}{0.5em}}
\numberwithin{equation}{section} \numberwithin{theorem}{section}

\newcommand{\Q }{\mathbb{Q}}
\newcommand{\Z }{\mathbb{Z}}
\newcommand{\N }{\mathbb{N}}

\newcommand{\W}{{\cal W}}

\newcommand{\E}{{\cal E}}

\begin{document}

\title{\textbf{Soft $\Gamma$-Semirings}}
\author{ \"{O}. Bekta\c{s}, N. Bayrak, B. A. Ersoy \\
obektas@yildiz.edu.tr, nbayrak@yildiz.edu.tr, ersoya@yildiz.edu.tr \\
Yildiz Technical University, Faculty of Arts and Sciences\\
Department of Mathematics, Istanbul, TURKEY}

\maketitle
\begin{abstract}
In this paper, the definitions of soft $\Gamma$-semirings and soft
sub $\Gamma$-semi rings are introduced with the aid of the concept
of soft set theory introduced by Molodtsov. In the mean time, some
of their properties and structural characteristics are
investigated and discussed. Thereafter, several illustrative
examples are given.

\textbf{Keywords}: Soft Sets, Fuzzy Sets, Semirings,
$\Gamma$-ring. \
\end{abstract}

\section{Introduction}\label{Sec1}
Uncertain data modelling was investigated by many researchers in
economics, engineering, environmental sciences, sociology, medical
science and many other fields. The process in classical
mathematics may not be competent owing to the fact that the
assorted uncertainties deriving in these fields. In this context,
mathematical theories such as probability theory, fuzzy set theory
\cite{Zadeh}, rough set theory \cite{Pawlak} were established by
researchers to modelling uncertainties arising in the stated
fields. In 1999, Molodtsov \cite{Molodtsov} made a new viewpoint
of substantial theoretical approaches: the concept of soft set
theory which is more convenient than classical ideologies and can
be seen as a outstanding mathematical tool relates with
uncertainties. After Molodtsov's work, some different applications
of soft sets were studied in \cite {Chen,Maji1,Maji2,Ali}.

The algebraic structure of soft set theories has been studied
progressively in recent years. Akta\c{s} and \c{C}a\u{g}man
\cite{Aktaþ} investigated basic properties of soft sets to the
related concepts of fuzzy sets and rough sets. They also defined
the notion of soft groups, and derived some related properties.
Furthermore, Maji et al. \cite{Maji3,Roy} presented the definition
of fuzzy soft set. The concept of  fuzzy soft groups which is a
generalization of soft groups were given in \cite{Aygun} and
\cite{Manemaran}.In 2010, a tentative approach between fuzzy sets
(rough sets) and soft sets were studied by Feng et al. in
\cite{FengFeng}.

On the other hand soft rings, soft ideals on soft rings and
idealistic soft rings were defined in \cite{Jun1}. After these
studies the notion of fuzzy soft rings and fuzzy soft ideals were
discussed in \cite{Jayanta}. In addition to this in \cite{Kazancý}
the concept of soft BCH-algebra was introduced and some of their
properties and structural characteristics were mentioned.

Furthermore the notion of soft semirings are investigated in
\cite{Feng} which is useful for dealing with problems in different
areas of applied mathematics and information sciences. The
semiring structure provides an algebraic framework for modelling
and investigating the key factors in these problems. Then, N.
Nobusawa \cite{Nobusawa} introduced the notion of $\Gamma$-ring,
as more general than ring. After that, the weakened conditions of
the definition of the $\Gamma$-ring were studied in \cite{Barnes}.
Then the generalization of $\Gamma$-ring and $\Gamma$-semiring
were introduced by \cite{Rao}.

Thereafter, in \cite{Jun2} Jun and Lee studied the concept of
fuzzy $\Gamma$-ring and \cite{Ozturk} defined the soft
$\Gamma$-rings and idealistic soft $\Gamma$-rings with their basic
properties. The extension of the $\Gamma$-semiring to quasi ideals
was done by \cite {Dutta,Jagatap,Chinram} with incompatible style.

In this paper, we introduce the concept of soft $\Gamma$-semiring
which extend the notion of soft $\Gamma$-ring theory and deal with
some of its algebraic properties by giving several examples.

\section{Soft $\Gamma$-Semirings }\label{Sec2}

\begin{definition}
 A pair $\left(\rho,W\right)$ is called a soft set
over $V$, where $\rho$ is a mapping from $W$ to $P\left(V\right)$
\cite{Molodtsov}.

\end{definition}

\begin{definition}
Let  $\left(\rho,W\right),\left(\sigma,Y\right)$ be soft sets over
a common universe $V$.
\begin{itemize}

\item [i)] If $W\subseteq Y$ and
$\rho\left(\omega\right)\subseteq\sigma\left(\omega\right)$  for
all $\omega\in W$ then we say that $\left(\rho,W\right)$ is a soft
subset of $\left(\sigma,Y\right)$,denoted by
$\left(\rho,W\right)\widetilde{\subseteq }\left(\sigma,Y\right)$.

\item [ii)] If $\left(\rho,W\right)$ is a soft subset of
$\left(\sigma,Y\right)$ and $\left(\sigma,Y\right)$ is a soft
subset of $\left(\rho,W\right)$, then we say that
$\left(\rho,W\right)$ is soft equal to $\left(\sigma,Y\right)$,
denoted by $\left(\rho,W\right)=\left(\sigma,Y\right)$ .
\end{itemize}
\end{definition}

\begin{definition}
\begin{itemize}

\item [i)] Let $\left(\rho,W\right)$ and $\left(\sigma,Y\right)$
be two soft set over a common universe $V$.

\[\left(\psi,Z\right)=\left(\rho,W\right)\widetilde{\cap}_\Re\left(\sigma,Y\right)\]

is said to be restricted-intersection of $\left(\rho,W\right)$ and
$\left(\sigma,Y\right)$, where $\left(\psi,Z\right)$ is soft set,
$Z=W\cap Y\neq \emptyset$ and the mapping $\psi$ is defined by

\[\begin{array}{clll}
  \psi:& Z & \rightarrow & P\left(V\right) \\
       & z & \rightarrow & \psi\left(z\right)=\rho \left(z\right)\cap \sigma\left(z\right) \\
\end{array}\]

\item [ii)] Let $\left\{\left(\rho_i,W_i\right): i\in I \right\}$
be non-empty family soft sets. The restricted-intersection of a
non-empty family soft sets is defined by

\[\left(\psi,Y\right)=\left(\widetilde{\cap}_\Re\right)_{i\in I}\left(\rho_i,W_i\right)\]
where $\left(\psi,Y\right)$ is a soft set, $Y=\bigcap_{i\in
I}W_i\neq\emptyset$ and $\psi\left(y\right)=\bigcap_{i\in
I}\rho_i\left(y\right)$ for every $y\in Y$
\cite{Feng},\cite{Ali}.
\end{itemize}
\end{definition}

\begin{definition}
\begin{itemize}
\item [i)] Let $\left(\rho,W\right)$ and
$\left(\sigma,Y\right)$ be two soft set over a common universe
$V$.

\[\left(\psi,Z\right)=\left(\rho,W\right)\widetilde{\cap}_\E\left(\sigma,Y\right)\]

is  called extended-intersection of $\left(\rho,W\right)$ and
$\left(\sigma,Y\right)$, where $\left(\psi,Z\right)$ is soft set
and $\left(\psi,Z\right)$ satisfying the following conditions

\item $Z=W\cup Y$
\item $\psi\left(z\right)=\left \{\begin{array}{llc}
                                                             \rho\left(z\right)                        &, if\; z\in W\backslash Y  \\
                                                             \sigma\left(z\right)                      &, if \; z\in Y\backslash W  \\
                                                             \rho\left(z\right)\cap\sigma\left(z\right)&,if  \; z\in c\in W\cap Y.\\
                                                            \end{array}\right.$

\item [ii)]Let $\left\{\left(\rho_i,W_i\right): i\in I \right\}$
be non-empty family soft sets. The extended-intersection of a
non-empty family soft sets is defined by

\[\left(\psi,Y\right)=\left(\widetilde{\cap}_\E\right)_{i\in I}\left(\rho_i,W_i\right)\]
where $\left(\psi,Y\right)$ is a soft set, $Y=\bigcup_{i\in
I}W_i,\psi\left(y\right)=\bigcap_{i\in I}\rho_i\left(y\right)$ and
$I\left(y\right)= \left\{i:i\in W_i\right\}$ for every $y\in Y$
\cite{Ali},\cite{Feng} .
\end{itemize}
\end{definition}

\begin{definition}
\begin{itemize}
\item [i)] Let $\left(\rho,W\right)$ and $\left(\sigma,Y\right)$
be two soft set over a common universe $V$.

\[\left(\psi,Z\right)=\left(\rho,W\right)\widetilde{\cup}_\Re\left(\sigma,Y\right)\]

is  said to be restricted union of $\left(\rho,W\right)$ and
$\left(\sigma,Y\right)$, where $\left(\psi,Z\right)$ is soft set,
$Z=W\cap Y\neq \emptyset$, and the mapping $\psi$ is defined by

\[\begin{array}{clll}
  \psi:& Z & \rightarrow & P\left(V\right) \\
       & z & \rightarrow & \psi\left(z\right)=\rho \left(z\right)\cup \sigma\left(z\right) \\
\end{array}\] \cite{Ali}.

\item [ii)]Let $\left\{\left(\rho_i,W_i\right): i\in I \right\}$
be non-empty family soft sets. The restricted-union of a non-empty
family soft sets is defined by

\[\left(\psi,Y\right)=\left(\widetilde{\cup}_\Re\right)_{i\in I}\left(\rho_i,W_i\right)\]
where $\left(\psi,Y\right)$ is a soft set, $Y=\bigcap_{i\in
I}W_i\neq\emptyset$ and $\psi\left(y\right)=\bigcup_{i\in
I}\rho_i\left(y\right)$ for every $y\in Y$ \cite{Kazancý}.

\end{itemize}
\end{definition}

\begin{definition}
\begin{itemize}

\item [i)]Let $\left(\rho,W\right)$ and $\left(\sigma,Y\right)$
be two soft set over a common universe $V$.
\[\left(\psi,Z\right)=\left(\rho,W\right)\widetilde{\cup}_\E\left(\sigma,Y\right)\]
is said to be extended union of $\left(\rho,W\right)$ and
$\left(\sigma,Y\right)$ where $\left(\psi,Z\right)$ is a soft set,
$Z=W\cup Y$, and the mapping $\psi$ is defined by
$$\psi\left(z\right)=\left \{\begin{array}{llc}
                                                             \rho\left(z\right)                        &, if  \;z\in W\backslash Y  \\
                                                             \sigma\left(z\right)                      &, if  \;z\in Y\backslash W  \\
                                                             \rho\left(z\right)\cup\sigma\left(z\right)&, if  \;z\in c\in W\cap Y.\\
                                                            \end{array}\right.$$

\item [ii)]Let $\left\{\left(\rho_i,W_i\right): i\in I \right\}$
be non-empty family soft sets. The extended-union of a non-empty
family soft sets is defined by

\[\left(\psi,Y\right)=\left(\widetilde{\cup}_\E\right)_{i\in I}\left(\rho_i,W_i\right)\]
where $\left(\psi,Y\right)$ is a soft set, $Y=\bigcup_{i\in I}W_i,
\psi\left(y\right)=\bigcup_{i\in I}\rho_i\left(y\right)$ and
$I\left(y\right)=\left\{i:i\in W_i\right\}$ for every $y\in Y$
\cite{Ali}.

\end{itemize}
\end{definition}

\begin{definition}
\begin{itemize}
\item [i)] Let $\left(\rho,W\right)$ and $\left(\sigma,Y\right)$
be two soft set over a common universe $V$.

\[\left(\psi,Z\right)=\left(\rho,W\right)\widetilde{\Lambda}\left(\sigma,Y\right)\]

is  called $\Lambda$-intersection of $\left(\rho,W\right)$ and
$\left(\sigma,Y\right)$, where $\left(\psi,Z\right)$ is soft set,
$Z=W\times Y$ and $\psi\left(w,y\right)=\rho
\left(\omega\right)\cap\sigma\left(y\right)$ for every
$\left(\omega,y\right) \in W\times Y$.

\item [ii)]Let $\left\{\left(\rho_i,W_i \right): i\in I\right\}$ be non-empty
family soft sets. The $\Lambda$-intersection of a non-empty family
soft sets is defined by

\[\left(\psi,Y\right)=\widetilde{\Lambda}_{i\in I}\left(\rho_i,W_i\right)\]
where $\left(\psi,Y\right)$ is a soft set, $Y=\Pi_{i\in I}W_i$ and
$\psi\left(y\right)=\bigcap_{i\in I}\rho_i\left(y\right)$ for
every $y=\left(y_i\right)_{i\in I}\in Y$ \cite{Maji2,Feng}.
\end{itemize}
\end{definition}

\begin{definition}
\begin{itemize}
\item [i)] Let $\left(\rho,W\right)$ and $\left(\sigma,Y\right)$
be two soft set over a common universe $V$.

\[\left(\psi,Z\right)=\left(\rho,W\right)\widetilde{\vee}\left(\sigma,Y\right)\]

is  called $\vee$-union of $\left(\rho,W\right)$ and
$\left(\sigma,Y\right)$, where $\left(\psi,Z\right)$ is soft set,
$Z=W\times Y$ and $\psi\left(w,y\right)=\rho
\left(\omega\right)\cup\sigma\left(y\right)$ for every
$\left(\omega,y\right)\in W\times Y$.

\item [ii)]Let $\left\{\left(\rho_i,W_i: i\in I \right)\right\}$ be non-empty
family soft sets. The $\vee$-union of a non-empty family soft sets
is defined by

\[\left(\psi,Y\right)=\widetilde{\vee}_{i\in I}\left(\rho_i,W_i\right)\]
where $\left(\psi,Y\right)$ is a soft set, $Y=\Pi_{i\in I}W_i$ and
$\psi\left(y\right)=\bigcup_{i\in I}\rho_i\left(y\right)$ for
every $y=\left(y_i\right)_{i\in I}\in Y$ \cite{Maji2,Feng}.
\end{itemize}
\end{definition}

\begin{definition}
\begin{itemize}

\item [i)]Let $\left(\rho,W\right)$ and $\left(\sigma,Y\right)$ be two soft
sets over a common universe $V_1$ and $V_2$ respectively. The
cartesian product of two soft sets $\left(\rho,W\right)$ and
$\left(\sigma,Y\right)$ is defined by
\[\left(Z, W\times Y\right)=\left(\rho,W\right)\times
\left(\sigma,Y\right)\] where $\left(Z, W\times Y\right)$ is a
soft set, and $\psi\left(\omega,y\right)=\rho
\left(\omega\right)\times \sigma \left(y\right)$ for every
$\left(w,y\right)\in W\times Y$\cite{Maji2}.

\item [ii)]Let $\left\{\left(\rho_i,W_i\right): i\in I \right\}$ be non-empty
family soft sets over $V_i,i\in I$. The cartesian product of a
non-empty family soft sets $\left\{\left(\rho_i,W_i\right): i\in I
\right\}$ over the universes $V_i$, is defined by

\[\left(\psi,Y\right)=\widetilde{\Pi}_{i\in I}\left(\rho_i,W_i\right)\]
where $\left(\psi,Y\right)$ is a soft set, $Y=\Pi_{i\in I}W_i$ and
$\psi\left(y\right)=\Pi_{i\in I}{\rho_i\left(y\right)}$ for all
$y=\left(y_i\right)_{i\in I}\in Y$ \cite{Kazancý}.
\end{itemize}
\end{definition}

\begin{definition}
\begin{itemize}
\item [i)] Let $\left(\rho,W\right)$
be soft set over a common universe $V$. Then $\left(\rho,W\right)$
is said to be a relative null soft set, denoted by $N_W$, if
$\rho\left(e\right)=\emptyset$ for every $e\in W.$

\item [ii)] $\left(\rho,W\right)$ is said to be relative whole
soft, denoted by $\W_W$, if $\rho\left(e\right)=V$ for every $e\in
W$ \cite{Kazancý}.
\end{itemize}
\end{definition}

\begin{definition}
Let $\left(\rho,W\right)$ and $\left(\sigma,Y\right)$ be two softs
set over a common universe $V_1$ and $V_2$, respectively, and
$f:V_1\rightarrow V_2, g: W\rightarrow Y$ be two functions.
$\left(f,g\right)$ is said to be a soft function from
$\left(\rho,W\right)$ to $\left(\sigma,Y\right)$, denoted by
$\left(f,g\right):\left(\rho,W\right)\rightarrow\left(\sigma,Y\right)$
if the following condition
\[f\left(\rho\left(\omega\right)\right)=\sigma
\left(g\left(\omega\right)\right)\] satisfies for all $w\in W$. If
$f$ and $g$ are injective (resp. surjective, bijective), then we
say that $\left(f,g\right)$ is injective(resp. surjective,
bijective)\cite{Kazancý}.
\end{definition}

\begin{lemma}
Let $\left(\rho,W\right),\left(\sigma,Y\right)$ and
$\left(\psi,Z\right)$  be soft sets over $V_1,V_2$ and $V_3$,
respectively. If
$$\left(f,g\right):\left(\rho,W\right)\rightarrow\left(\sigma,Y\right)$$
and
$$\left(f',g'\right):\left(\sigma,Y\right)\rightarrow\left(\psi,Z\right)$$
are two soft functions, then

$$\left(f'\circ f,g'\circ g\right):\left(\rho,W\right)\rightarrow
\left(\psi,Z\right)$$ is a soft function.
\end{lemma}

\begin{definition}
Let $\left(\rho,W\right)$ and $\left(\sigma,Y\right)$ be two soft
sets over $V_1$ and $V_2$,respectively, $\left(f,g\right)$ is a
soft function from $\left(\rho,W\right)$ to
$\left(\sigma,Y\right)$ .The image of $\left(\rho,W\right)$ under
the soft function $\left(f,g\right)$, denoted by
$\left(f,g\right)\left(\rho,W\right)=\left(
f\left(\rho\right),Y\right)$, is the soft set over $V_1$ defined
by
$$f\left(\rho\right)\left(y\right)=\left\{\begin{array}{llc}
                                                             V_1{_{g\left(\omega\right)=y}}f\left(\rho\left(\omega\right)\right)               & if  y\in Img \\
                                                             \emptyset                                                        & otherwise  \\
                                                            \end{array}\right.$$

for all $y\in Y$. The pre-image of $\left(\sigma,Y\right)$ under
the soft function $\left(f,g\right)$ denoted by
$\left(f,g\right)^{-1}\left(\sigma,Y\right)=\left(f^{-1}\left(\sigma\right),W\right)$,
is the soft set over $V_1$ defined by
$f^{-1}\left(\sigma\right)\left(\omega\right)=f^{-1}\left(\sigma\left(\rho\left(\omega\right)\right)\right)$
for all $\omega\in W$.

It is clear that $\left(f,g\right)\left(\rho,W\right)$ is a soft
subset of $\left(\sigma,Y\right)$ and $\left(\rho,W\right)$ is a
soft subset of $\left(f,g\right)^{-1}\left(\sigma,Y\right)$. In
particular, if $\rho$ is the identity function on $W$, the soft
sets $\left(f\left(\rho\right),W\right)$ and
$f^{-1}\left(\sigma\right)\left(\omega\right)$ are as given in
\cite{Aktaþ} and \cite{Kazancý}.
\end{definition}

\begin{definition}
Let $S$ and $\Gamma$ be two additive commutative semigroups. Then
$S$ is called a $\Gamma$- semiring if there exists a mapping
$S\times\Gamma\times S$ (images to be denoted by $a \alpha b$ for
all $a,b\in S$ and $\alpha\in \Gamma$) satisfying the following
conditions

\begin{itemize}
\item [i)]$\left(a+b\right)\alpha c=a\alpha c+b \alpha c$

\item [ii)]$a\alpha\left(b+c\right)=a\alpha b+a \alpha c$

\item [iii)]$a\left(\alpha+\beta\right)b=a\alpha b+a \beta b$

\item [iv)]$a\alpha\left(b\beta c \right)b= \left(a \alpha b
\right)\beta c$ for all $a,b,c \in S$ and for all $\alpha,\beta\in
\Gamma$ \cite{Jagatap}.
\end{itemize}
\end{definition}

\begin{example}
Let $\Q$  be set of rational numbers. $\left(S,+\right)$ be the
commutative semigroup of all $2\times 3$ matrices over $\Q$ and
$\left(\Gamma,+\right)$ be commutative semigroup of all $3\times
2$ matrices over $\Q$. Define $W\alpha Y$ usual matrix product of
$W,\alpha$ and $Y$ ; for all $W,Y\in S$ and for all $\alpha\in
\Gamma$. Then $S$ is a $\Gamma$-semiring but not a semiring
\cite{Jagatap}.
\end{example}

\begin{remark}
Let $\N$ be the set of natural numbers and
$\Gamma=\left\{1,2,3\right\}$. Define the mapping
$\N\times\Gamma\times\N\rightarrow\N$ by $a\alpha b= a\cdot \alpha
\cdot b$ (usual product of $a,\alpha$ and $b$); for all $a,b \in
\N,\alpha \in \Gamma $. Then $\N$ is a $\Gamma$-semiring given in
\cite{Chinram}. But $\Gamma$ is not an additive semigroup, hence
it is not a $\Gamma$-semiring according to \cite{Dutta}.
\end{remark}

\begin{example}
Let $\N$  be the set of natural numbers and
$\Gamma=\left\{1,2,3\right\} \left(\N, max\right)$ and
$\left(\Gamma, max\right)$ are commutative semigroups. Define the
mapping $\N\times\Gamma\times\N\rightarrow\N$ , by $a\alpha b= min
\left\{a, \alpha, b\right\}$, for all $a,b \in \N,\alpha\in
\Gamma$. Then $\N$ is a $\Gamma$- semiring \cite{Jagatap}.
\end{example}

\begin{example}
Let $\Q$  be set of rational numbers and $\Gamma=\N$ be the set of
natural numbers $\left(\Q,+\right)$ and $\left(\N,+\right)$ are
commutative semigroups. Define the mapping
$\Q\times\N\times\Q\rightarrow\Q$ by $a\alpha b$ usual product of
$a, \alpha, b;a,b \in \mathbb{Q},\alpha \in \Gamma $. Then $\Q$ is
a $\Gamma$-semiring \cite{Jagatap}.
 \end{example}

\section{Soft $\Gamma$-Semiring }\label{Sec3}

Let $S$ be a nonempty set and a $\Gamma$-semiring. $R$ will allude
to any triplet relation the midst of a component of $S$ and a
component of $\Gamma$ and a component of $S$, that is,
esoterically $R$ is a subset of $S\times\Gamma\times S$. In this
way, a set valued function $\psi:N \rightarrow P\left(S\right)$
can be defined as

\[\psi\left(y\right)=\left\{s\in S: R\left(y, \alpha,s\right),\forall \alpha\in \Gamma\right\}\]

for all $y\in N$. The pair $\left(\psi, N\right)$ is then a soft
set over $S$, which produced from the relation $R$. The set

\[Supp \left(\psi, N, \Gamma\right)= \left\{y\in N: \psi\left(y\right)\neq\emptyset\right\}\]
is called a  support of the soft set $\left(\psi,N,\Gamma\right)$.
The soft set $\left(\psi,N,\Gamma\right)$ non-null if $Supp
\left(\psi, N, \Gamma\right)\neq \emptyset$ \cite{Kazancý,Feng}.

\begin{definition}
A nonempty subset $T$ of $S$ is said to be a sub-$\Gamma$-
semiring of $S$ if $\left(T,+\right)$ is a subsemigroup of
$\left(S,+\right)$ and $a\alpha b\in T$; for all $a,b\in T$ and
for all $\alpha\in \Gamma$ \cite{Jagatap}.
\end{definition}

\begin{definition}
Let  $\left(\psi,N\right)$ be a non-null soft set over a
$\Gamma$-semiring $S$. Then  $\left(\psi,N\right)$ is called a
soft $\Gamma$- semiring over $S$ if $ \psi\left(x\right)$ is a
sub-$\Gamma$-semiring of $S$ for all $y\in Supp
\left(\psi,N\right) $. This denoted by
$\left(\psi,N,\Gamma\right)$.
\end{definition}

\begin{example}
For consider the additively abelian groups
$\Z_8=\left\{0,1,2,3,4,5,6,7\right\}$ and $\Gamma=
\left\{2,4,6\right\}$. Let $\cdot: \mathbb{Z}_8\times \Gamma\times
\Z_8 \rightarrow \Z_8, \left(y, \alpha, s\right)=y\alpha s$.
Therefore we have that
\begin{itemize}

\item [i)] $\left(a+b\right)\alpha c=a\alpha c+b \alpha c$

\item [ii)] $a\alpha\left(b+c\right)=a\alpha b+a \alpha c$

\item [iii)]$a\left(\alpha+\beta\right)b=a\alpha b+a \beta b$

\item [iv)]$a\alpha\left(b\beta c \right)b= \left(a \alpha b
\right)\beta c$ for all $a,b,c \in \mathbb{Z}_8$ and for all
$\alpha,\beta\in \Gamma=\left\{2,4,6\right\}$. Hence $\Z_8$ is a
$\Gamma$- semiring.

Let $N=\Z_8$ and $\psi:N \rightarrow P\left(\Z_8\right)$ be a set
valued function defined by
\[\psi\left(y\right)=\left\{s\in \Z_8 : R\left(y, \alpha,s\right)\leftrightarrow, \left(y, \alpha,s\right)\in \left\{0,4,6\right\}, \forall \alpha\in \Gamma\right\}\]
for all $y\in N=\Z_8$. Then

$$\psi\left(0\right)=\psi\left(2\right)=\psi\left(4\right)=\psi\left(6\right)=\Z_8$$
$$\psi\left(1\right)=\psi\left(3\right)=\psi\left(5\right)=\psi\left(7\right)=\left\{0,2,4,6\right\}$$
 are sub-$\Gamma$-semirings of $\Z_8$. Hence
 $\left(\psi,N\right)$ is a soft-$\Gamma$- semiring over
 $\mathbb{Z}_8$.
\end{itemize}
\end{example}

\begin{proposition}
Let $\left(\rho,W,\Gamma\right)$ and
$\left(\sigma,W,\Gamma\right)$ be soft $\Gamma$ semirings over
$\Gamma$-semiring $S$. The restricted intersection
$\left(\rho,W,\Gamma\right)\widetilde{\cap}_\Re\left(\sigma,W,\Gamma\right)$
is a soft $\Gamma$semiring over $S$ if it is non-null.
\end{proposition}

\begin{proof}
By Definition 2.3 (i), we have that
$\left(\rho,W,\Gamma\right)\widetilde{\cap}_\Re\left(\sigma,W,\Gamma\right)=\left(\psi,W,\Gamma\right)$
where $\psi\left(\omega\right)=\rho\left(\omega\right)\cap
\sigma\left(\omega\right)$ for all $\omega\in W$. We assume that
$\left(\psi,W,\Gamma\right)$ is a non-null soft set over $S$. If
$\omega\in  Supp \left(\psi,W,\Gamma\right)$, then
$\psi\left(\omega\right)=\rho\left(\omega\right)\cap
\sigma\left(\omega\right)\neq \emptyset$. We know that
$\left(\rho,W,\Gamma\right)$ and $\left(\sigma,W,\Gamma\right)$
are both soft $\Gamma$ semirings over $S$, and so, the nonempty
sets $\rho\left(\omega\right)$ and $\sigma\left(\omega\right)$ are
both sub $\Gamma$ semiring of $S$ (From definition 3.2). Thus,
$\psi\left(w\right)$ is a sub $\Gamma$-semiring of $S$ for all
$\omega\in Supp \left(\psi,W,\Gamma\right)$. In this position,
$\left(\psi,W,\Gamma\right)=
\left(\rho,W,\Gamma\right)\widetilde{\cap}_\Re\left(\sigma,W,\Gamma\right)$
is a soft $\Gamma$ semiring over $S$
\end{proof}

\begin{corollary}
Let $ \left\{\left(\rho_i,W,\Gamma\right):i\in I\right\}$ be a
nonempty family of soft $\Gamma$-semiring over $S$. Then the
restricted intersection $\left(\widetilde{\cap}_\Re\right)_{i\in
I}\left(\rho_i,W,\Gamma\right)$ is a soft $\Gamma$-semiring over
$S$ if it is non-null.
\end{corollary}

\begin{proof}
Straight forward
\end{proof}

\begin{theorem}
Let $\left(\rho_i,W_i,\Gamma\right)_{i\in I}$ be a nonempty family
of soft-$\Gamma$-semirings over $S$. Then the restricted
intersection $\left(\widetilde{\cap}_\Re\right)_{i\in
I}\left(\rho_i,W_i,\Gamma\right)$ is a soft $\Gamma$-semiring over
$S$ if it is non-null.
\end{theorem}

\begin{proof}
From definition 2.3(ii), we have that
$\left(\widetilde{\cap}_\Re\right)_{i\in
I}\left(\rho_i,W_i,\Gamma\right)= \left(\psi,Y, \Gamma\right)$,
where $Y=\bigcap_{i\in I} W_i\neq \emptyset$, and
$\psi\left(y\right)=\bigcap _{i\in I} \rho_i \left(y\right)$ for
every $y\in Y$.

We assume that $\left(\psi,Y, \Gamma\right)$ is non-null. Let
$y\in Supp \left(\psi,Y, \Gamma\right)$. Then $\psi
\left(y\right)\neq\emptyset$ and so we have
$\rho_i\left(y\right)\neq\emptyset$ for every $i\in I$. From the
hypothesis, we know that
$\left\{\left(\rho_i,W_i,\Gamma\right):i\in I\right\}$ is a
nonempty family of soft-$\Gamma$-semiring over $S$, by definition
3.2 $\rho_i \left(y\right)$ is a sub-$\Gamma$-semiring of $S$,
that is, $\psi\left(y\right)$ is a sub-$\Gamma$-semiring of $S$
for all $y\in Supp \left(\psi,Y, \Gamma\right) $ and so
$\left(\psi,Y, \Gamma\right)$ is a soft $\Gamma$ semiring over
$S$.
\end{proof}

\begin{theorem}
Let $ \left\{\left(\rho_i,W_i,\Gamma\right):i\in I\right\}$ be a
nonempty family of soft $\Gamma$-semiring over $S$. Then the
extended intersection $\left(\widetilde{\cap}_\E\right)_{i\in
I}\left(\rho_i,W_i,\Gamma\right)$ is a soft $\Gamma$-semirings
over $S$.
\end{theorem}

\begin{proof}
From definition 2.4 (ii), we have that
$\left(\widetilde{\cap}_\E\right)_{i\in
I}\left(\rho_i,W_i,\Gamma\right)=\left(\psi,Y, \Gamma\right)$
where $Y=\bigcup_{i\in I} W_i$, and
$\psi\left(y\right)=\bigcap_{i\in I}\rho_i\left(y\right)$ for all
$y\in Y$.

Assume that $y\in Supp \left(\psi,Y, \Gamma\right)$. Then
$\psi\left(y\right)\neq \emptyset$ and so we have
$\rho_i\left(y\right)\neq \emptyset$ for every $i\in I$. Because
of the fact that$ \left\{\left(\rho_i,W_i,\Gamma\right):i\in
I\right\}$ is a soft $\Gamma$-semiring over $S$ for every $i\in
I$, we have that $\rho_i\left(y\right)$ is a sub $\Gamma$-semiring
over $S$ for every $i\in I$. It follows that
$\psi\left(y\right)=\bigcap_{i\in I}\rho_i\left(y\right)$ is a
sub-$\Gamma$-semiring over $S$ for every $y\in \left(\psi,Y,
\Gamma\right)$. Thus, $\left(\widetilde{\cap}_\E\right)_{i\in
I}\left(\rho_i,W_i,\Gamma\right)$ is a soft-$\Gamma$-semiring over
$S$.
\end{proof}

\begin{theorem}
Let $ \left\{\left(\rho_i,W_i,\Gamma\right):i\in I\right\}$ be a
nonempty family of soft $\Gamma$-semirings over $S$. If
$\rho_i\left(y_i\right)\subseteq \rho_j\left(y_j\right)$ or
$\rho_j\left(y_j\right)\subseteq \rho_i\left(y_i\right) $ for all
$i,j \in I, y_i \in W_i$ then the restricted union
$\left(\widetilde{\cup}_\Re\right)_{i\in
I}\left(\rho_i,W_i\right)$ is a soft-$\Gamma$-semiring over $S$.
\end{theorem}

\begin{proof}
Using definition 2.5 (ii), we have that
$\left(\widetilde{\cup}_\Re\right)_{i\in
I}\left(\rho_i,W_i,\Gamma\right)=\left(\psi,Y, \Gamma\right)$
where $Y=\bigcap_{i\in I} W_i$, and
$\psi\left(y\right)=\bigcup_{i\in I}\rho_i\left(y\right)$ for all
$y\in Y$.
Assume that $y\in Supp \left(\psi,Y, \Gamma\right)$.
Then $\psi\left(y\right)\neq \emptyset$ and so we have
$\rho_{i_0}\left(y\right)\neq \emptyset$ for some $i_0\in
I\left(y\right)$. By assumption, $\bigcup_{i\in
I}\rho_i\left(y\right)$ is a sub $\Gamma$-semiring of $S$ for
every $y\in Supp \left(\psi,Y, \Gamma\right)$. Hence,
$\left(\widetilde{\cup}_\Re\right)_{i\in
I}\left(\rho_i,W_i,\Gamma\right)$ is a soft-$\Gamma$-semiring over
$S$.
\end{proof}

\begin{theorem}
Let $ \left\{\left(\rho_i,W_i,\Gamma\right):i\in I\right\}$ be a
nonempty family of soft $\Gamma$-semiring over $S$. Let $W_i$ and
$W_j$ be members of the family $\left\{W_i: i\in I\right\}$ such
that $W_i\cap W_j= \emptyset$ for $i\neq j$. Then
$\left(\widetilde{\cup}_\E\right)_{i\in
I}\left(\rho_i,W_i,\Gamma\right)$ is a soft-$\Gamma$-semiring over
$S$.
\end{theorem}

\begin{proof}
From definition 2.6 (ii) we have that where
$\left(\widetilde{\cup}_\E\right)_{i\in
I}\left(\rho_i,W_i,\Gamma\right)=\left(\psi,Y, \Gamma\right)$
$\psi\left(y\right)=\bigcap_{i\in I}{\rho_i\left(y\right)}$ for
all $y\in Y$. Note first that $\left(\psi,Y\right)$ is non-null
owing to the fact that $Supp \left(\psi,Y,
\Gamma\right)=\bigcup_{i\in I}Supp \left(\rho_i,W_i,
\Gamma\right)$. Suppose that $y\in Supp \left(\psi,Y,
\Gamma\right)$. Then $\psi\left(y\right)\neq \emptyset$ so we have
$\rho_{i_0}\neq\emptyset$ for some $i_0\in I\left(y\right)$. From
the hypothesis $\left\{W_i: i\in I\right\}$ are pairwise disjoint,
we follow that $\varphi\left(y\right)=\rho_{i_0}\left(y\right)$.
On the other hand $\rho_{i_0}\left(y\right)$ is a soft
$\Gamma$-semiring over $S$, we conclude that $\left(\psi,Y\right)$
is a soft $\Gamma$-semiring over $S$ for all $y\in
\left(\psi,Y,\Gamma\right)$. Consequently
$\left(\widetilde{\cup}_\E\right)_{i\in
I}\left(\rho_i,W_i,\Gamma\right)= \left(\psi,Y, \Gamma\right)$ is
a soft $\Gamma$-semiring over S.
\end{proof}

\begin{theorem}
If $\left(\rho,W,\Gamma\right)$ and $\left(\sigma,Y,\Gamma\right)$
be two soft $\Gamma$-semirings over $\Gamma$-semiring $S$, then
$\left(\rho,W,\Gamma\right)
\widetilde{\Lambda}\left(\sigma,Y,\Gamma\right)$ is a soft
$\Gamma$-semiring over S if it is non-null.
\end{theorem}

\begin{proof}
Using definition 2.7 (i) , we have that
$\left(\rho,W,\Gamma\right)
\widetilde{\Lambda}_\E\left(\sigma,Y,\Gamma\right)=\left(\psi,Z,\Gamma\right)$,
where $Z=W \times \Gamma \times Y$ and
$\psi\left(\omega,\alpha,y\right)= \rho\left(\omega\right)\cap
\sigma\left(y\right)$ for all $\left(\omega,\alpha,y\right)\in Z=
W\times \Gamma \times Y$. Then by the hypothesis,
$\left(\psi,Z,\Gamma\right)$ is a nonnull soft set over
$\Gamma$-semiring $S$. Since $\left(\psi,Z,\Gamma\right)$ is a
nonnull, $Supp\left(\psi,Z,\Gamma\right)\neq \emptyset$ and so,
for $\left(\omega,\alpha,y\right)\in
Supp\left(\psi,Z,\Gamma\right), \psi\left(\omega,\alpha,y\right)=
\rho\left(\omega\right)\cap \sigma\left(y\right)\neq \emptyset$.
We assume that $t_1, t_2\in \rho\left(\omega\right)\cap
\sigma\left(y\right)$. In this position

\begin{itemize}
\item [i)]If $t_1, t_2\in\rho\left(\omega\right)=\left\{y: R\left(\omega,\alpha_1,y\right),\forall \alpha_1\in
\Gamma\right\}$ we have that $\omega\alpha_1t_1\in
W,\omega\alpha_1t_2\in W$. This implies
$\omega\alpha_1\left(t_1+t_2\right)\in W$ and

\item [ii)]$t_1, t_2\in\sigma\left(y\right)=\left\{y': R\left(y,\alpha_2,y'\right),\forall \alpha_2\in
\Gamma\right\}$ we have that $y\alpha_2t_1\in Y,y\alpha_2t_1\in
Y.$ This implies $y\alpha_2\left(t_1+t_2\right)\in Y$.

Hence $\rho\left(x\right)\cap\sigma\left(y\right)$ is a
sub-$\Gamma$ semiring. By definition of soft $\Gamma$ semiring,
$\left(\rho,W,\Gamma\right)$ and $\left(\sigma,Y,\Gamma\right)$
are both soft $\Gamma$ semirings over $S$. $\rho\left(x\right)$
and $\sigma\left(y\right)$ are also sub-$\Gamma$ semiring of $S$.
Furthermore
$\psi\left(\omega,\alpha,y\right)=\rho\left(\omega\right)\cap\sigma\left(y\right)$
is a sub $\Gamma$ semiring of $S$ for all
$\left(\omega,\alpha,y\right)\in
\left(\psi,Z,\Gamma\right)=\left(\rho,W,\Gamma\right)\widetilde{\wedge}\left(\sigma,Y,\Gamma\right)$
is a soft $\Gamma$ semiring over $S$ required.
\end{itemize}
\end{proof}

\begin{theorem}
Let $ \left\{\left(\rho_i,W_i,\Gamma\right):i\in I\right\}$ be a
nonempty family of soft $\Gamma$-semiring over $S$. Then
$\widetilde{\Lambda}_{i\in I}\left(\rho_i,W_i,\Gamma\right)$ is a
soft $\Gamma$-semiring over $S$ if it is non-null.
\end{theorem}

\begin{proof}
By taking into account to the definition 2.7 (ii) we write
$\widetilde{\wedge}_{i\in
I}\left(\rho_i,W_i,\Gamma\right)=\left(\psi,Y, \Gamma\right)$,
where $Y=\prod_{i\in I}W_i$, and $\psi\left(y\right)=\bigcap_{i\in
I}\rho_i \left(y\right)$ for all $y=\left(y_i\right)_{i\in I}\in
Y$.

Suppose that $\left(\psi,Y, \Gamma\right)$ is non-null. If
$y=\left(y_i\right)_{i\in I}\in Supp \left(\psi,Y, \Gamma\right)$,
then $\psi\left(y\right)\neq \emptyset$. Since
$\left(\rho_i,W_i,\Gamma\right)$ is a soft $\Gamma$-semiring over
$S$ for all $i\in I$ members of nonempty family
$\left(\left(\rho_i,W_i,\Gamma\right):i\in I\right)$ such that
$\rho_i \left(y_i\right)$ is a sub $\Gamma$-semiring of $S$. Hence
$\psi\left(y\right)$ is a sub $\Gamma$-semiring of $S$ for all
$y\in Supp \left(\psi,Y, \Gamma\right) $, and so
$\widetilde{\wedge}_{i\in
I}\left(\rho_i,W_i,\Gamma\right)=\left(\psi,Y, \Gamma\right)$ is
soft $\Gamma$-semiring over $S$.
\end{proof}

\begin{theorem}
Let $ \left\{\left(\rho_i,W_i,\Gamma\right):i\in I\right\}$ be a
nonempty family of soft $\Gamma$-semiring over $S$. If
$\rho_i\left (y_i\right)\subseteq \rho_j\left (y_j\right)$ or
$\rho_j\left (y_j\right) \subseteq \rho_i\left (y_i\right)$ for
all $i,j\in I,y_i\in W_i$, the $\vee$-union
$\widetilde{\vee}_{i\in I}\left(\rho_i,W_i,\Gamma\right)$ is a
soft $\Gamma$-semiring over $S$.
\end{theorem}

\begin{proof}
Using the definition 2.8 (ii), we have that
$\widetilde{\vee}_{i\in
I}\left(\rho_i,W_i,\Gamma\right)=\left(\psi,Y, \Gamma\right)$
where $Y=\prod_{i\in I} A_i$, and
$\psi\left(y\right)=\bigcup_{i\in I}\rho_i\left(y\right)$ for all
$y=\left(y_i\right)_{i}\in I\in Y$.

Assume that $y=\left(y_i\right)_{i}\in I\in Supp \left(\psi,Y,
\Gamma\right)$. Then $\psi\left(y\right)\neq \emptyset$ and so we
have that $\rho_{i_0}\left(y\right)\neq \emptyset$ for some
$i_0\in I$. By assumption, $\bigcup_{i\in I}\rho_i\left(y\right)$
is a soft $\Gamma$-semiring of $S$ for all
$y=\left(y_i\right)_{i}\in I\in Supp \left(\psi,Y, \Gamma\right)$.
Consequently $\widetilde{\vee}_{i\in
I}\left(\rho_i,W_i,\Gamma\right)=\left(\psi,Y, \Gamma\right)$ is a
soft-$\Gamma$-semiring over $S$
\end{proof}

\begin{theorem}
Let $ \left\{\left(\rho_i,W_i,\Gamma\right):i\in I\right\}$ be a
nonempty family of soft $\Gamma$-semirings over $S_i$.Then
$\widetilde{\prod}_{i\in I}\left(\rho_i,W_i,\Gamma\right)$ is a
soft $\Gamma$-semiring over $\prod_{i\in I}S_i$.
\end{theorem}

\begin{proof}
By definition 2.10 we write $\widetilde{\prod}_{i\in
I}\left(\rho_i,W_i,\Gamma\right)=\left(\psi,Y, \Gamma\right)$,
where $Y=\prod_{i\in I}W_i$, and $\psi\left(y\right)=\prod_{i\in
I}W_i \left(y\right)$ for all $y=\left(y_i\right)_{i\in I}\in Y$.

Let $y=\left(y_i\right)_{i\in I}\in Supp \left(\psi,Y,
\Gamma\right)$. Then $\psi\left(y\right)\neq \emptyset$, and so we
have $\rho_i\left(y_i\right)\neq \emptyset$ for all $i\in I$. By
taking into account,$ \left\{\left(\rho_i,W_i,\Gamma\right):i\in
I\right\}$ is a soft $\Gamma$-semiring over $S_i$ for all $i\in
I$, it follows that $\prod_{i\in I}\rho_i\left(y_i\right)$ is a
soft-$\Gamma$-semiring of $\prod_{i\in I}S_i$ for all
$y=\left(y_i\right)_{i\in I}\in Supp \left(\psi,Y,
\Gamma\right)$.Hence $\widetilde{\prod}_{i\in
I}\left(\rho_i,W_i,\Gamma\right)$ is a soft $\Gamma$-semiring over
$\prod_{i\in I} S_i$
\end{proof}

\begin{definition}
Let $\left(\rho,W,\Gamma\right)$ be soft $\Gamma$-semiring over
$S$.
\begin{itemize}
\item [i)] $\left(\rho,W,\Gamma\right)$ is called the trivial soft
$\Gamma$-semiring over $S$ if
$\rho\left(\omega\right)=\left\{0\right\}$ for all $\omega\in W$

\item [ii)] $\left(\rho,W,\Gamma\right)$ is called the whole soft
$\Gamma$-semiring over $S$ if $\rho\left(\omega\right)=S$ for all
$\omega\in W$
\end{itemize}
\end{definition}
\begin{definition}
Let $S$ and $S'$ be two  $\Gamma$-semiring and $f: S\rightarrow
S'$ a mapping of $\Gamma$-semiring. If $\left(\rho,W\right)$ and
$\left(\sigma,Y\right)$ are soft sets over $S$ and $S'$
respectively, then

\begin{itemize}
\item [i)]
$\left(f\left(\rho\right),W\right)$ is a soft set over $S'$ where

$$f\left(\rho\right): W\rightarrow P\left(S'\right)$$

$$f\left(\rho\right)\left(\omega\right)=f\left(\rho\left(w\right)\right)$$

for all $\omega\in W$.

\item [ii)]
$\left(f^{-1}\left(\sigma\right),Y\right)$ is a soft set over $S$
where
$$f^{-1}\left(\sigma\right): Y\rightarrow
P\left(S\right)$$
$$f^{-1}\left(\sigma\right)\left(y\right)=f^{-1}\left(\sigma\left(y\right)\right)$$ for all $y \in Y$.
\end{itemize}
\end{definition}

\begin{lemma}
Let $f:S\rightarrow S'$ be  an onto homomorphism of
$\Gamma$-semiring. The following statements can be given.
\begin{itemize}
\item [i)] $\left(\rho,W,\Gamma\right)$ be soft $\Gamma$-semiring over
$S$, then  $\left(f\left(\rho\right),W,\Gamma\right)$ is a soft
$\Gamma$-semiring over $S'$

\item [ii)]  $\left(\sigma,Y,\Gamma\right)$ be soft $\Gamma$-semiring over
$S$, then  $\left(f^{-1}\left(\sigma\right),Y,\Gamma\right)$ is a
soft $\Gamma$-semiring over $S$.
\end{itemize}
\end{lemma}

\begin{proof}
\begin{itemize}
\item [i)] Since $\left(\rho,W,\Gamma\right)$  is a soft $\Gamma$-semiring over
$S$, it is clear that $\left(f\left(\rho\right),W\right)$ is a
non-null soft set over $S'$. For every $y\in Supp
\left(f\left(\rho\right),W,\Gamma\right)$  we have
$f\left(\rho\right)\left(y\right)=f\left(\rho\left(y\right)\right)\neq
\emptyset$. Hence $f\left(\rho\left(y\right)\right)$ which is the
onto homomorphic image of  $\Gamma$-semiring $\rho\left(y\right)$
is a $\Gamma$-semiring of $S'$ for all $y\in Supp
\left(\rho\left(f\right),W,\Gamma\right)$. That is
$\left(f\left(\rho\right),W,\Gamma\right)$  is a soft
$\Gamma$-semiring of $S'$.

\item [ii)]It is easy to see that $ Supp\left(f^{-1}\left(\sigma\right),Y,\Gamma\right)\subseteq Supp
\left(\sigma,Y,\Gamma\right)$. By this way let $y\in Supp
\left(f^{-1}\left(\sigma\right),Y,\Gamma\right)$. Then
$\sigma\left(y\right)\neq \emptyset$. Hence
$f^{-1}\left(\sigma\left(y\right)\right)$ which is homomorphic
inverse image of $\Gamma$-semiring $\sigma\left(y\right)$, is a
soft $\Gamma$-semiring over $S$ for all $y\in Y$.
\end{itemize}
\end{proof}

\begin{theorem}
Let $f:S\rightarrow S'$ be a homomorphism of $\Gamma$-semiring.
Let $\left(\rho,W,\Gamma\right)$ and
$\left(\sigma,Y,\Gamma\right)$ be two soft $\Gamma$-semiring over
$S$ and $S'$, respectively. Then the following sare given.
\begin{itemize}

\item [i)] If $\rho\left(\omega\right)=ker \left(f\right)$ for all $\omega\in W$,
then $\left(f\left(\rho\right),W, \Gamma\right)$ is the trivial
soft $\Gamma$-semiring over $S'$.

\item [ii)]If $f$ is onto and $\left(\rho,W\right)$ is whole, then $\left(f\left(\rho\right),W, \Gamma\right)$ is
the whole soft $\Gamma$-semiring over $S'$.

\item [iii)]If $\sigma\left(y\right)=f \left(S\right)$ for all $y\in Y$, then
$\left(f^{-1}\left(\sigma\right),Y, \Gamma\right)$ is the whole
soft $\Gamma$-semiring over $S$.

\item [iv)]If f is injective and $\left(\sigma,Y\right)$ is trivial, then
$\left(f^{-1}\left(\sigma\right),Y, \Gamma\right)$ is the trivial
soft $\Gamma$-semiring over $S$.
\end{itemize}
\end{theorem}

\begin{proof}
\begin{itemize}
\item [i)] By using $\rho\left(\omega\right)=ker \left(f\right)$ for all $y\in W$. Then $f\left(\rho\right)\left(\omega\right)= f \left(\rho\left(\omega\right)\right)=\left\{0_{S'}\right\}$
for all $\omega\in W$. Hence $\left(f\left(\rho\right),W,
\Gamma\right)$ is soft $\Gamma$-semiring over $S'$ by Lemma 3.16
and Definition 3.14.

\item [ii)] Suppose that $f$ is onto and  $\left(\rho,W\right)$ is
whole. Then $\rho\left(\omega\right)=S$ for all $\omega\in W$, and
so $f\left(\rho\right)\left(\omega\right)= f
\left(\rho\left(\omega\right)\right)=f \left(S\right)=S'$ for all
$\omega\in W$. It follows from Lemma 3.16 and Definition 3.14 that
$\left(f\left(\rho\right),W\right)$ is the whole soft
$\Gamma$-semiring over $S'$.

\item [iii)]If we use hypothesis $\sigma\left(y\right)=f\left(S\right)$ for all $y\in
Y$, we can write $f^{-1}\left(\sigma\right)\left(y\right)= f^{-1}
\left(\sigma\left(y\right)\right)=f^{-1}
\left(f\left(S\right)\right)=S$ for all $y\in Y$. It is clear
that, $\left(f^{-1}\left(\sigma\right),B,\Gamma\right)$ is the
whole soft $\Gamma$-semiring over $S$ by Lemma 3.16 and Definition
3.14.

\item [iv)]Suppose that f is injective and $\left(\sigma,Y\right)$ is trivial.
Then, $\sigma\left(y\right)=\left\{0\right\}$ for all $y\in Y$, so
$f^{-1}\left(\sigma\right)\left(y\right)=f^{-1}
\left(\sigma\left(y\right)\right)=f^{-1}\left(\left\{0\right\}\right)=ker
f=\left\{0_S\right\}$ for all $y\in Y $. It follows from Lemma
3.16 and Definition 3.14 that
$\left(f^{-1}\left(\sigma\right),Y,\Gamma\right)$ is the trivial
soft $\Gamma$-semiring over $S$.
\end{itemize}
\end{proof}

\section{Soft Sub $\Gamma$-Semiring }\label{Sec3}

\begin{definition}
Let $\left(\rho,W,\Gamma\right)$  and
$\left(\sigma,Y,\Gamma\right)$ be two soft $\Gamma$-semirings over
$S$. Then the soft $\Gamma$-semiring is called a soft sub
$\Gamma$-semiring of $\left(\rho,W,\Gamma\right)$, denoted by
$\left(\sigma,Y,\Gamma\right)\subset_{\Gamma_s}\left(\rho,W,\Gamma\right)$,
if it satisfies the following conditions

\item [i)] $Y\subseteq W$,

\item [ii)] $\sigma\left(y\right)$ is a sub $\Gamma$-Semiring of
$\rho\left(y\right)$ for all $y\in Supp
\left(\sigma,Y,\Gamma\right)$.

From the above definition, it is easily deduced that if
$\left(\sigma,Y,\Gamma\right)$ is a soft sub $\Gamma$-Semiring of
$\left(\rho,W,\Gamma\right)$, then $Supp
\left(\sigma,Y,\Gamma\right)\subset
Supp\left(\rho,W,\Gamma\right)$.
\end{definition}
\begin{theorem}
Let $\left(\rho,W,\Gamma\right)$ and
$\left(\sigma,Y,\Gamma\right)$ be two soft $\Gamma$-semirings over
$S$ and
$\left(\rho,W,\Gamma\right)\widetilde{\subseteq}\left(\sigma,Y,\Gamma\right)$.
Then
$\left(\sigma,Y,\Gamma\right)\subset_{\Gamma_s}\left(\rho,W,\Gamma\right)$,
\end{theorem}
\begin{proof}
Straightforward.
\end{proof}

\begin{theorem}
Let $\left(\rho,W,\Gamma\right)$ and
$\left(\sigma,Y,\Gamma\right)$ be two soft $\Gamma$-semirings over
$S$ and
$\left(\rho,W,\Gamma\right)\widetilde{\sqcap}\left(\sigma,Y,\Gamma\right)$
is a soft sub $\Gamma$ semiring of both
$\left(\rho,W,\Gamma\right)$ and $\left(\sigma,Y,\Gamma\right)$ if
it is non-null.
\end{theorem}
\begin{proof}
Straightforward.
\end{proof}
\begin{theorem}
Let $\left(\rho,W,\Gamma\right)$ be soft $\Gamma$-semiring over
$S$ and $\left\{\left(\psi_i,W_i,\Gamma\right): i\in I\right\}$ be
nonempty family of soft sub $\Gamma$-semirings of
$\left(\rho,W,\Gamma\right)$. Then the restricted intersection
$\left(\widetilde{\cap}_\Re\right)_{i\in
I}\left(\psi_i,W_i,\Gamma\right)$ is a soft sub $\Gamma$-semiring
of $\left(\rho,W,\Gamma\right)$ if it is non-null.
\end{theorem}
\begin{proof}
Similar to the proof of Theorem 3.6.
\end{proof}

\begin{corollary}
Let $\left(\rho,W,\Gamma\right)$ be soft $\Gamma$-semiring over
$S$ and $\left\{\left(\psi_i,W,\Gamma\right): i\in I\right\}$ be
nonempty family of soft sub $\Gamma$-semirings of
$\left(\rho,W,\Gamma\right)$. Then
$\left(\widetilde{\cap}_\Re\right)_{i\in
I}\left(\psi_i,W,\Gamma\right)$ is a soft sub $\Gamma$-semiring of
$\left(\rho,W,\Gamma\right)$ if it is non-null.
\end{corollary}

\begin{proof}
Straightforward.
\end{proof}

\begin{theorem}
Let $\left(\rho,W,\Gamma\right)$ be soft $\Gamma$-semiring over
$S$ and $\left\{\left(\psi_i,W_i,\Gamma\right): i\in I\right\}$ be
nonempty family of soft sub $\Gamma$-semirings of
$\left(\rho,W,\Gamma\right)$. Then the extended intersection
$\left(\widetilde{\cap}_E\right)_{i\in
I}\left(\psi_i,W_i,\Gamma\right)$ is a soft sub $\Gamma$-semiring
of $\left(\rho,W,\Gamma\right)$.
\end{theorem}
\begin{proof}
Similar to the proof of Theorem 3.7.
\end{proof}
\begin{theorem}
Let $\left(\rho,W,\Gamma\right)$ be soft $\Gamma$-semiring over
$S$ and $\left\{\left(\psi_i,W_i,\Gamma\right): i\in I\right\}$ be
nonempty family of soft sub $\Gamma$-semirings of
$\left(\rho,W,\Gamma\right)$. If $\psi_i\left(y_i\right)\subseteq
\psi_j\left(y_j\right)$ or $\psi_j\left(y_j\right)\subseteq
\psi_i\left(y_i\right)$ for all $i,j\in I, y_i\in W_i$, then the
restricted union $\left(\widetilde{\cup}_\Re\right)_{i\in
I}\left(\psi_i,W_i,\Gamma\right)$ is a soft sub $\Gamma$-semiring
of $\left(\rho,W,\Gamma\right)$.
\end{theorem}

\begin{proof}
By the aid of the definition 2.6 (ii), we  write
$\left(\widetilde{\cup}_\E\right)_{i\in
I}\left(\psi_i,W_i,\Gamma\right)=\left(\psi,Y,\Gamma\right)$,
where $Y=\bigcup_{i\in I}W_i$, and
$\psi\left(y\right)=\bigcup_{i\in I}\psi_i\left(y\right)$ for all
$y\in Y$.

Let $y\in Supp \left(\psi,Y,\Gamma\right) $. Then
$\psi\left(y\right)\neq \emptyset$, and so we have
$\psi_{i_0}\left(y_{i_0}\right)\neq \emptyset$ for some $i_0 \in
I$.From the hypothesis, we know that
$\psi_i\left(y_i\right)\subseteq \psi_j\left(y_j\right)$ or
$\psi_j\left(y_j\right)\subseteq \psi_i\left(y_i\right)$ for all
$i,j\in I, y_i\in W_i$, clearly $\bigcup_{i\in
I}\psi_i\left(y\right)$ is a sub $\Gamma$- semiring of
$\rho\left(y\right)$ for all $y\in Supp
\left(\psi,Y,\Gamma\right)$. Thus
$\left(\widetilde{\cup}_\Re\right)_{i\in
I}\left(\psi_i,W_i,\Gamma\right)=\left(\psi,Y,\Gamma\right)$ is a
soft sub $\Gamma$-semiring of $\left(\rho,W,\Gamma\right)$.
\end{proof}

\begin{theorem}
Let $\left(\rho,W,\Gamma\right)$ be soft $\Gamma$-semiring  over
$S$ and $\left\{\left(\psi_i,W_i,\Gamma\right): i\in I\right\}$ be
nonempty family of soft sub $\Gamma$-semiring of
$\left(\rho,W,\Gamma\right)$. If $\psi_i\left(y_i\right)\subseteq
\psi_j\left(y_j\right)$ or $\psi_j\left(y_j\right)\subseteq
\psi_i\left(y_i\right)$ for all $i,j\in I, y_i\in W_i$,then $\vee$
union $\widetilde{\vee}_{i\in I}\left(\psi_i,W_i,\Gamma\right)$ is
a soft sub $\Gamma$-semiring of $\widetilde{\vee}_{i\in
I}\left(\rho,W,\Gamma\right)$.
\end{theorem}

\begin{proof}
Similar to the proof of Theorem 3.12.
\end{proof}

\begin{theorem}
Let $\left(\rho,W,\Gamma\right)$ be a soft $\Gamma$-semiring over
$S$ and $\left\{\left(\psi_i,W_i,\Gamma\right): i\in I\right\}$ be
nonempty family of soft sub $\Gamma$-semirings of
$\left(\rho,W,\Gamma\right)$. Then the $\wedge$ intersection
$\widetilde{\wedge}_{i\in I}\left(\psi_i,W_i,\Gamma\right)$ is a
soft sub $\Gamma$-semiring of $\widetilde{\wedge}_{i\in
I}\left(\rho,W,\Gamma\right)$.
\end{theorem}

\begin{proof}
Similar to the proof of Theorem 3.11.
\end{proof}

\begin{theorem}
Let $\left(\rho,W,\Gamma\right)$ be soft $\Gamma$-semiring over
$S$ and $\left\{\left(\psi_i,W_i,\Gamma\right): i\in I\right\}$ be
nonempty family of soft sub $\Gamma$-semirings of
$\left(\rho,W,\Gamma\right)$. Then the cartesian product of the
family $\widetilde{\prod}_{i\in I}\left(\psi_i,W_i,\Gamma\right)$
is a soft sub $\Gamma$-semiring of $\widetilde{\prod}_{i\in
I}\left(\rho,W,\Gamma\right)$.
\end{theorem}

\begin{proof}
By Definition 2.10, we can write $\widetilde{\prod}_{i\in
I}\left(\psi_i,W_i,\Gamma\right)=\left(\psi,Y,\Gamma\right)$ where
$Y=\prod_{i\in I}W_i$ and $\psi\left(y\right)=\prod_{i\in
I}\psi_i\left(y_i\right) $ for all $y=\left(y_i\right)_{i\in I}\in
Y$. Let $y=\left(y_i\right)_{i\in I}\in Supp
\left(\psi,Y,\Gamma\right)$. Then $\psi\left(y\right)\neq
\emptyset$ and so we have $\psi_i\left(y_i\right)\neq \emptyset$
for all $i\in I$. In as much as
$\left\{\left(\psi_i,W_i,\Gamma\right): i\in I\right\}$ is a soft
sub $\Gamma$-semiring of $\left(\psi,W,\Gamma\right)$, we have
that $\psi_i\left(y_i\right)$ is a sub $\Gamma$-semiring of
$\rho\left(y_i\right)$. It follows that, we obtain  $\prod_{i\in
I}\psi_i\left(y_i\right)$ for all $y=\left(y_i\right)_{i\in I}\in
Supp \left(\psi,Y,\Gamma\right)$. Hence, the cartesian product of
the family $\widetilde{\prod}_{i\in
I}\left(\rho_i,W_i,\Gamma\right)$ is a soft sub $\Gamma$-semiring
of $\left(\rho,W,\Gamma\right)$.
\end{proof}

\begin{theorem}
Let $f:S\rightarrow S'$ be a homomorphism of $\Gamma$-semirings
and $\left(\rho,W,\Gamma\right)$ and
$\left(\sigma,Y,\Gamma\right)$ two soft $\Gamma$-semirings over
$S$. If
$\left(\sigma,Y,\Gamma\right)\subset_{\Gamma_S}\left(\rho,W,\Gamma\right)$
then $\left(f\left(\sigma\right),Y,\Gamma\right)
\subset_{\Gamma_S}\left(f\left(\rho\right), W,\Gamma\right)$.
\end{theorem}

\begin{proof}
Suppose that $y\in Supp \left(\sigma,Y, \Gamma\right)$. Then $y\in
Supp \left(\rho,W, \Gamma\right) $. By definition 4.1, we know
that $ Y\subseteq W$ and $\sigma\left(y\right)$ is a sub
$\Gamma$-semiring of $\rho\left(y\right)$ for all $y\in Supp
\left(\sigma,Y, \Gamma\right)$. From the expression hypothesis $f$
is a homomorphism, $f\left(\sigma\right)\left(y\right)=
f\left(\sigma\left(y\right)\right)$ is a sub $\Gamma$-semiring of
$
f\left(\rho\right)\left(y\right)=f\left(\rho\left(y\right)\right)$
and therefore $\left(f\left(\sigma\right),Y,
\Gamma\right)\subset_{\Gamma_S}\left(f\left(\rho\right),W,
\Gamma\right)$.
\end{proof}

\begin{theorem}
Let $f:S\rightarrow S'$ be a homomorphism of $\Gamma$-semiring and
$\left(\rho,W,\Gamma\right)$, $\left(\sigma,Y,\Gamma\right)$ two
soft $\Gamma$-semirings over $S$. If
$\left(\sigma,Y,\Gamma\right)\subset_{\Gamma_S}\left(\rho,W,\Gamma\right)$
then $\left(f^{-1}\left(\sigma\right),Y,\Gamma\right)
\subset_{\Gamma_S}\left(f^{-1}\left(\rho\right), W,\Gamma\right)$.
\end{theorem}

\begin{proof}
Let $y\in Supp \left(f^{-1}\left(\sigma\right),Y, \Gamma\right)$.
$Y\subseteq W$ and $\sigma\left(y\right)$ is a sub
$\Gamma$-semiring of $\rho\left(y\right)$ for all $y\in Y$. Since
$f$ is a homomorphism, $f^{-1}\left(\sigma\right)\left(y\right)=
f^{-1}\left(\sigma\left(y\right)\right)$ is a sub
$\Gamma$-semiring of
$f^{-1}\left(\sigma\left(y\right)\right)=f\left(\sigma\right)\left(y\right)$
for all $y\in  Supp
\left(f^{-1}\left(\sigma\right),Y,\Gamma\right)$. Hence
$\left(f^{-1}\left(\sigma\right),Y,\Gamma\right)\subset_{\Gamma_S}\left(f^{-1}\left(\rho\right),
W,\Gamma\right)$
\end{proof}
\begin{definition}
Let $\left(\rho,W, \Gamma\right)$ and $\left(\sigma,Y,
\Gamma\right)$ be two soft $\Gamma$-semiring over $S$ and $S'$,
respectively. Let $f: S\rightarrow S'$ and $f: W\rightarrow Y$ be
two functions. The following conditions:
\begin{itemize}
\item [i)] $f$ is an epimorphism of $\Gamma$-semiring

\item [ii)] $g$ is and surjective mapping.

\item [ii)]
$f\left(\rho\left(y\right)\right)=\sigma\left(\rho\left(y\right)\right)$
for all $y\in W$.

were satisfied by the pair $\left(f,g\right)$, then
$\left(f,g\right)$ is called soft $\Gamma$- semiring homomorphism.
\end{itemize}

If there exists a soft $Gamma$-semiring homomorphism between
$\left(\rho, W, \Gamma\right)$ and $\left(\sigma, Y,
\Gamma\right)$,we say that $\left(\rho, W, \Gamma\right)$ is soft
homomorphic to $\left(\sigma, Y, \Gamma\right)$, and is denoted by
$\left(\rho, W, \Gamma\right)\sim_{\Gamma_s}\left(\sigma, Y,
\Gamma\right)$ is soft isomorphic to $\left(G, B, \Gamma\right)$,
which is denoted by $\left(\rho, W, \Gamma\right)\simeq
_{\Gamma_S}\left(\sigma, Y, \Gamma\right)$.
\end{definition}

\end{document}